%% file: main.tex
\documentclass{article}
\input{preamble.tex}
\input{title}

\begin{document}
\maketitle

\begin{abstract}
\input{abstract}
\end{abstract}

\section{Introduction}\label{sec:intro}
\input{introduction}

\section{Notation}\label{sec:notation}
\input{notation}

\section{Preliminaries}\label{sec:prelim}
\input{prelim}

\section{Projection Formula}\label{sec:proj}
\input{proj}

\section{Generalized Wedderburn Rank Reduction}
\input{generalized_wedd}

\section{Wedderburn Decomposition}
\input{wedd_decomp}

\section{Best Rank Reduction via Generalized Wedderburn Rank Reduction}
\input{best_rank}

\section{Properties Inherited by Generalized Wedderburn Rank Reduction}
\input{misc}

\section{Examples}
\input{matlab}

\bibliographystyle{unsrt}  
\bibliography{references}  

\end{document}

%% file: preamble.tex
\usepackage{arxiv}

\usepackage[utf8]{inputenc} 
\usepackage[T1]{fontenc}    
\usepackage{hyperref}       
\usepackage{url}            
\usepackage{booktabs}       
\usepackage{amsfonts}       
\usepackage{nicefrac}       
\usepackage{microtype}      
\usepackage{lipsum}
\usepackage{mathtools,physics}
\usepackage{enumerate}
\usepackage{graphicx}
\graphicspath{ {./images/} }
\usepackage{amsmath}
\usepackage{amssymb}
\usepackage{amsthm}
\usepackage{stackengine} 
\stackMath
\usepackage[framed,autolinebreaks,useliterate]{mcode} 
\usepackage{centernot}

\numberwithin{equation}{section}

\newtheorem{theorem}{Theorem}[section]
\newtheorem{lemma}[theorem]{Lemma}
\newtheorem{corollary}[theorem]{Corollary}
\newtheorem{remark}[theorem]{Remark}
\newtheorem{example}[theorem]{Example}

\newcommand{\rng}[1]{\mathcal{C}\left(#1\right)}
\newcommand{\nullsp}[1]{\mathcal{N}\left(#1\right)}
\newcommand{\pinv}[1]{{#1}^+\!}

\newcommand{\proj}[2]{{#1}\pinv{\brac*{{#2}^*{#1}}}{#2}^*}

\DeclarePairedDelimiter\brac{(}{)}

\newcommand{\bK}{\mathbb{K}}
\newcommand{\bC}{\mathbb{C}}
\newcommand{\bR}{\mathbb{R}}

\makeatletter
\let\save@mathaccent\mathaccent
\newcommand*\if@single[3]{%
  \setbox0\hbox{${\mathaccent"0362{#1}}^H$}%
  \setbox2\hbox{${\mathaccent"0362{\kern0pt#1}}^H$}%
  \ifdim\ht0=\ht2 #3\else #2\fi
  }
\newcommand*\rel@kern[1]{\kern#1\dimexpr\macc@kerna}
\newcommand*\widebar[1]{\@ifnextchar^{{\wide@bar{#1}{0}}}{\wide@bar{#1}{1}}}
\newcommand*\wide@bar[2]{\if@single{#1}{\wide@bar@{#1}{#2}{1}}{\wide@bar@{#1}{#2}{2}}}
\newcommand*\wide@bar@[3]{%
  \begingroup
  \def\mathaccent##1##2{%
    \let\mathaccent\save@mathaccent
    \if#32 \let\macc@nucleus\first@char \fi
    \setbox\z@\hbox{$\macc@style{\macc@nucleus}_{}$}%
    \setbox\tw@\hbox{$\macc@style{\macc@nucleus}{}_{}$}%
    \dimen@\wd\tw@
    \advance\dimen@-\wd\z@
    \divide\dimen@ 3
    \@tempdima\wd\tw@
    \advance\@tempdima-\scriptspace
    \divide\@tempdima 10
    \advance\dimen@-\@tempdima
    \ifdim\dimen@>\z@ \dimen@0pt\fi
    \rel@kern{0.6}\kern-\dimen@
    \if#31
      \overline{\rel@kern{-0.6}\kern\dimen@\macc@nucleus\rel@kern{0.4}\kern\dimen@}%
      \advance\dimen@0.4\dimexpr\macc@kerna
      \let\final@kern#2%
      \ifdim\dimen@<\z@ \let\final@kern1\fi
      \if\final@kern1 \kern-\dimen@\fi
    \else
      \overline{\rel@kern{-0.6}\kern\dimen@#1}%
    \fi
  }%
  \macc@depth\@ne
  \let\math@bgroup\@empty \let\math@egroup\macc@set@skewchar
  \mathsurround\z@ \frozen@everymath{\mathgroup\macc@group\relax}%
  \macc@set@skewchar\relax
  \let\mathaccentV\macc@nested@a
  \if#31
    \macc@nested@a\relax111{#1}%
  \else
    \def\gobble@till@marker##1\endmarker{}%
    \futurelet\first@char\gobble@till@marker#1\endmarker
    \ifcat\noexpand\first@char A\else
      \def\first@char{}%
    \fi
    \macc@nested@a\relax111{\first@char}%
  \fi
  \endgroup
}
\makeatother

%% file: title.tex
\title{Generalized Wedderburn Rank Reduction}

\author{
  Oskar Kędzierski\\
  NASK National Research Institute\\
  Warsaw, Poland\\
  \texttt{oskar.kedzierski@nask.pl}\\
}

%% file: abstract.tex
We generalize the Wedderburn rank reduction formula by replacing the inverse with the Moore--Penrose pseudoinverse. In particular, this allows one to remove the non--singularity of a certain matrix from assumptions. The results implies in a straightforward way Nystroem, CUR decompositions, meta-factorization, and a result of Ameli, Shadden~\cite{AmeliShadden}. We investigate which properties of the matrix are inherited by the generalized Wedderburn reduction. Reductions leading to the best low-rank approximation are explicitly described in terms of singular vectors. We give a self--contained calculation of the range and the nullspace of the projection $A(BA)^+B$ and prove that any projection can be expressed in this way.

%% file: introduction.tex
The classical Wedderburn rank reduction formula, cf.~\cite[p.~69]{Wedderburn},\cite{ChuFunderlicGolub}, gives an explicit way to reduce the rank of a given matrix $A$, i.e.,
\begin{lemma}[Wedderburn rank reduction formula]\label{lem:classic_wedd}
Let $A\in\bK^{m\times n}$ be any matrix. Let $x\in\bK^n,\ y\in\bK^m$.
Assume that
\[\omega=y^* Ax\neq 0.\]
Then
\[\rank(A-\omega^{-1}Axy^*A)=\rank A -1.\]
In general, if $X\in\bK^{n\times k},\ Y\in\bK^{m\times k}$ and $M\in\bK^{k\times k}$ is given by
\[M=Y^* AX,\]
is invertible then 
\[\rank(A-AXM^{-1}Y^*A)=\rank A -k.\]
\end{lemma}

For the sake of completeness, we include a brief proof of both statements, although they follow from the more general result, cf.~Theorem~\ref{thm:gen_wedd}.

\begin{proof}
It follows from that assumptions that $Ax\neq 0$ and 
\[(A-\omega^{-1}Axy^*A)x=0.\]
Therefore 
\[\rank(A-\omega^{-1}Axy^*A)\le \rank A-1.\]
Assume now $(A-\omega^{-1}Axy^*A)u=0$. Let $\lambda=\omega^{-1}y^*Au$, then
\[A(u-\lambda x)=0,\]
hence $u\in \nullsp{A}+\bK x$. Therefore
\[\rank(A-\omega^{-1}Axy^*A)= \rank A-1.\] 
In the general case $\rank AX=k$ (otherwise $\rank M<k$) and the proof is similar.
That is 
\[\brac*{A-AXM^{-1}Y^*A}X=0,\]
and the $k$ columns of $X$ are linearly independent. On the other hand $AX\in\bK^{m\times k}$ is of rank $k$, so no linear combination of columns of $X$ is contained in the nullspace of $A$. Since $\nullsp{A}\subset\nullsp{A-AXM^{-1}Y^*A}$ 
\[\rank\brac*{A-AXM^{-1}Y^*A}=n-\dim\nullsp{A-AXM^{-1}Y^*A}\le n-\brac*{(n-\rank{A})+k}=\rank A - k.\]
Let $U\in\bK^{n\times k}$ be any matrix such that 
\[\brac*{A-AXM^{-1}Y^*A}U=0.\] 
Let $\Lambda=M^{-1}Y^*AU$ then
\[A(U-X\Lambda)=0,\]
hence
\[U\in\nullsp{A}+X\Lambda.\]
This implies that
\[\nullsp{A-AXM^{-1}Y^*A}\subset \nullsp{A}+\rng{X}\cong \nullsp{A}\oplus\rng{X},\]
and finally
\[\rank\brac*{A-AXM^{-1}Y^*A}=\rank A-k.\]
\end{proof}

The main result of the following paper is a generalization of the Wedderburn rank reduction by replacing the inverse with the Moore--Penrose pseudoinverse and dropping the non--singularity assumption; see Theorem~\ref{thm:gen_wedd}.
Results related to our generalization can be found in~\cite{MahdaviRaboky}. The result is used to prove the Wedderburn decomposition (Lemma~\ref{lem:wedd_decomp}), related to the generalized Nystr\"om method of Nakatsukasa~\cite{Nakatsukasa}. It generalizes the CUR decomposition (cf. \cite[Theorem~5.5~(ii),(iv)]{HammCU}), when matrices $X,Y$ have columns given by unit vectors corresponding to the chosen rows and columns. Another consequence is the so-called meta-factorization by Karpowicz\cite[Theorem~4]{karpowicz2022theory}. One of the main tools is the projection formula (see Lemma~\ref{lem:proj}), proven here from the first principles, without the Zlobec formula, cf.~\cite{Cerny}. Finally, matrices $X,Y$ which lead to the best $k$-rank generalized Wedderburn rank reduction are given in terms of SVD decomposition of matrix $A$, see Lemma~\ref{lem:SVD}. The properties shared by a matrix and its reduction are discussed in Remark~\ref{rem:inherited}, with concrete counterexamples. Theorem~\ref{thm:meet_and_join} contains a description of the meet and the join of commuting projections in the form as in Lemma~\ref{lem:proj}. The correctness of the results is illustrated by the MATLAB code.

%% file: notation.tex
Let $[n]=\{1,\ldots,n\}$ and $e_i=(0,\ldots,0,1,0,\ldots,0)^\intercal\in\bK^n$ for $i=1,\ldots,n$. The set of all matrices with $m$ rows, $n$ columns, and coefficients in field $\bK$ is denoted $\bK^{m\times n}$, where $\bK=\bR$ or $\bK=\bC$.  The column space of matrix $A$ is denoted by $\rng{A}$ and the null space by $\nullsp{A}$.  The Moore--Penrose pseudoinverse of matrix $A$ is denoted by $\pinv{A}$. The complex conjugate is denoted by $A^*$. Any matrix $A$ induces the direct sum decomposition $\bK^m=\rng{A}\oplus\nullsp{A^*}$. Matrix $A$ is Hermitian (or symmetric over $\bR$) if $A^*=A$. A square matrix is unitary (or orthogonal over $\bR$) if $A^*A=I$ where $I$ denotes the unit matrix. SVD stands for singular value decomposition, i.e., decomposition $A=U\Sigma V^*$, where $U,V$ are unitary matrices, and $\Sigma$ is a real non--negative generalized diagonal matrix with decreasing diagonal entries. If $\rank A=r$ then the first $r$ columns of matrix $U$ are called left singular vectors of $A$ and likewise the first $r$ columns of matrix $V$ are called right singular vectors of $A$. 

If $W=V\oplus U$ is a direct sum decomposition of a vector space $W$ then the linear function $\varphi\colon \bK^n\rightarrow \bK^n$ given by condition $\varphi(v+u)=u$, where $v\in V,u\in U$, is called a projection onto subspace $V$ along $U$. It is a well--known fact that an endomorphis $\varphi$ is a projection onto its image along its kerenel if and only if the matrix $P$ of $\varphi$ (relative to the same fixed basis in the domain and in the codomain) is idepotent, i.e., $P^2=P$. The space $\bK^n$ is always equipped with the standard inner product. Endomorphism $\varphi$ is an orthogonal projection if and only if its matrix relative to an orthonormal basis is idempotent and Hermitian, i.e., $P^*=P=P^2$.

%% file: prelim.tex
The following facts about Moore--Penrose pseudoinverses will be used in the paper tacitly.

\begin{gather}
    A\pinv{A}A=A, \label{eq:Penrose_cond1} \\
    \pinv{A}A\pinv{A}=\pinv{A}, \label{eq:Penrose_cond2} \\
    (A\pinv{A})^*=A\pinv{A}, \label{eq:Penrose_cond3} \\
    (X\pinv{A})^*=X\pinv{A}. \label{eq:Penrose_cond4}
\end{gather}

In particular, $A\pinv{A}$ and $\pinv{A}A$ are orthogonal projections and

\[\rng{A\pinv{A}}=\rng{A},\quad  \nullsp{A\pinv{A}}=\nullsp{A^*},\]
\[\rng{\pinv{A}A}=\rng{A^*},\quad  \nullsp{\pinv{A}A}=\nullsp{A}.\]

The following fact was proven by Greville, and it will be referred to as Greville's condition.
\begin{equation}\label{eq:Grevilles_cond}
    \pinv{(AB)}=\pinv{B}\pinv{A}\Longleftrightarrow \rng{BB^*A^*}\subset \rng{A^*}\quad\text{and}\quad\rng{A^*AB}\subset \rng{B}.
\end{equation}

\begin{lemma}\label{lem:cancel_in_rowsp_and_colsp}
Let $A\in\bK^{m\times n},\ B\in\bK^{n\times k}$ be two matrices. Then
\[\rng{AB}=\rng{ABB^*},\quad \nullsp{AB}=\nullsp{A^*AB}.\]
\end{lemma}

\begin{proof}
    The first equation is true because $\rng{BB^*}=\rng{B}$ and the second because $\rng{A}\cap\nullsp{A^*}=0$.
\end{proof}

\begin{lemma}\label{lem:null_equal_ranks}
    If $\rank A=\rank B=\rank AB$ then
    \[\nullsp{AB}=\nullsp{B}.\]
\end{lemma}
\begin{proof}
    Clearly 
    \[\rng{B^*A^*}\subset  \rng{B^*},\]
    but the dimensions are equal.
\end{proof}
\begin{corollary}\label{cor:Greville_cond}
    If $\rank A=\rank B=\rank AB$ then
    \[\pinv{(AB)}=\pinv{B}\pinv{A}\Longleftrightarrow \nullsp{B^*}=\nullsp{A}.\]
\end{corollary}
\begin{proof}
    The condition $\nullsp{B^*}\subset\nullsp{B^*A^*A}$ reduces to $\nullsp{B^*}\subset\nullsp{A}$ and the condition
$\nullsp{A}\subset\nullsp{ABB^*}$ reduces to $\nullsp{A}\subset\nullsp{B^*}$. Those are dual to the Greville's conditions~(\ref{eq:Grevilles_cond}).
\end{proof}

%% file: proj.tex
The classical result is well--known in the following form.

\begin{lemma}\label{lem:proj}
Let $A\in\bK^{m\times p},\ B\in\bK^{m\times q}$ be matrices such that 
\[\rng{A}\oplus\nullsp{B^*}=\bK^m.\]
Then the matrix of the projection onto $\rng{A}$ along $\nullsp{B^*}$ is given by
\[P=A\pinv{(B^*A)}B^*.\]
\end{lemma}
\begin{proof}
    Since
    \[P^2=A\brac*{\pinv{(B^*A)}B^*A\pinv{(B^*A)}}B^*=A\pinv{(B^*A)}B^*=P,\]
    $P$ is a projection. The assumption implies that $\rank A=\rank B=\rank B^*A$ and therefore
    $\rng{P}=\rng{A}$ as $\rng{P}\subset\rng{A}$ and they have the same dimensions. Similarly $\nullsp{P}=\nullsp{B^*}$, by Lemma~\ref{lem:null_equal_ranks}.
\end{proof}

It turns out that the assumptions on the matrices $A,B$ may be dropped and the range of $P$ will become smaller while the null space will become larger. The same result was obtained in~\cite{Cerny} by \v{C}ern{\'y} but here we do not use the Zlobec formula. 

\begin{lemma}\label{lem:proj_dim}
Let $A\in\bK^{m\times p},\ B\in\bK^{m\times q}$ be any matrices.
Let 
\[P=A\pinv{(B^*A)}B^*.\]
Then $P$ is a projection of $\rank P=\rank(B^*A)$ and
\[\rng{P}=\rng{AA^*B},\quad\nullsp{P}=\nullsp{A^*BB^*}.\]
\end{lemma}
\begin{proof}
    By  direct computation $P^2=P$. The rank of a projection is equal to its trace. Therefore,
    \[\Tr P=\Tr\brac*{ A\pinv{(B^*A)}B^*}=\Tr\brac*{ B^*A\pinv{(B^*A)}}=\rank(B^*A). \]
    First, we compute the null space of $P$. Assume 
    \[Px=A\pinv{(B^*A)}B^*x=0,\]
    by multiplying by $B^*$ on the left
    \[B^*A\pinv{(B^*A)}B^*x=0,\]
    that is, since $\nullsp{B^*A\pinv{(B^*A)}}=\nullsp{\brac{B^*A}^*}$ 
    \[B^*x\in\nullsp{A^*B}\Longrightarrow A^*BB^*x=0.\]
    Assume now $A^*BB^*x=0$. Multiplying this equation on the left by $\pinv{(A^*B)}$ gives
    \[\pinv{(A^*B)}A^*BB^*x=0,\]
    and therefore
    \[0=\pinv{(A^*B)}A^*BB^*x=\brac*{\pinv{(B^*A)}}^*\brac*{B^*A}^*B^*x=\brac*{\brac*{B^*A}\pinv{(B^*A)}}^*B^*x=\]
    \[=\brac*{B^*A}\pinv{(B^*A)}B^*x=B^*\brac*{A\pinv{(B^*A)}B^*}x=B^*Px.\]
    Multiplying the equation $B^*Px=0$ by $A\pinv{(B^*A)}$ on the left we see that $PPx=0$, that is, $Px=0$.
    Finally
    \[\nullsp{P}=\nullsp{A^*BB^*}.\]
    Now 
    \[P^*=B\pinv{\brac*{A^*B}}A^*,\]
    and by the previous conclusion 
    \[\nullsp{P^*}=\nullsp{B^*AA^*}.\]
    Since $\rng{P}\oplus\nullsp{P^*}$ is an orthogonal decomposition then
    \[\rng{P}=\rng{AA^*B}.\]
\end{proof}

\begin{example}
    If $A=B$, then
    \[P=A\pinv{\brac*{A^*A}}A^*=A\pinv{A}.\]
\end{example}

\begin{example}
    If $A$ and $B$ are of full row rank, then $A,B$ are surjective, hence $P=I$.
    
\end{example}

Note that if $P$ is a projection, then $\pinv{P}$ in general is not a projection.

\begin{lemma}
    Let $P\in\bK^{m\times m}$ be a projection. Then
\[\pinv{P}\text{ is a projection}\Longleftrightarrow \pinv{P}=P^*=P.\]
\end{lemma}

\begin{proof}
$(\Longrightarrow)$ 
By properties of pseudoinverse we have
\[\rng{\pinv{P}}=\rng{P^*},\quad \nullsp{\pinv{P}}=\nullsp{P^*}.\]
But $P^*$ is a projection too and projections are uniquely determined by the range and nullspace. Therefore, $\pinv{P}=P^*$. \\
$(\Longleftarrow)$  Let $P=U\Sigma V^*$ be an SVD decomposition of $P$. Then $\pinv{P}=P^*$ is equivalent to $V\pinv{\Sigma} U^*=V\Sigma^* U^* $, that is, $\Sigma^*=\pinv{\Sigma}$, which implies that all singular values of $P$
are equal to $1$. Therefore, $P$ is an orthogonal projection.
\end{proof}


The following fact can be found in~\cite[Corollary~5.6]{IpsenMeyer} or~\cite[Lemma~6.4.16~iii) or Fact~8.8.3~iii)]{Bernstein}

\begin{lemma}\label{lem:pinv_of_orth_projs}
    Let $P\in\bK^{n\times n},\ Q\in\bK^{n\times n}$ be orthogonal projections. Then
\[\pinv{\brac*{PQ}}=Q\pinv{(PQ)}P,\]
is a projection.
\end{lemma}


\begin{corollary}\label{cor:prescribed_proj}
    Assume that $V\oplus W=\bK^m$. The formula of projection onto $V$ along $W$ is
    \[P=\pinv{\brac*{P_{W^\perp}P_V}},\]
    where $P_{W^\perp}$ is an orthogonal projection onto $W^\perp$ and $P_V$ is an orthogonal projection onto $V$. 
\end{corollary}

\begin{proof}
    The claim follows directly from Lemma~\ref{lem:proj}.
\end{proof}

\begin{corollary}
Any projection $P\in \bK^{m\times m}$ can be written in the form $P=A\pinv{(B^*A)}B^*$. for some matrices $A\in\bK^{m\times p},\ B\in\bK^{m\times q}$. 
Moreover, the complementary projection $I-P$ is given by the formula
\[I-P=C\pinv{\brac*{DC}}D,\]
where
\[C=I-\brac*{BB^*A}\pinv{\brac*{BB^*A}},\quad D=I-\brac*{AA^*B}\pinv{\brac*{AA^*B}}.\]
\end{corollary}

\begin{lemma}\label{lem:rol_implies_proj}
    If the reverse order law holds for $B^*$ and $A$, that is, $\pinv{\brac*{B^*A}}=\pinv{A}\pinv{\brac*{B^*}}$ and
    \[P=A\pinv{(B^*A)}B^*,\]
    then $P=A\pinv{A}B\pinv{B}$ and it is a projection onto $\rng{A}\cap\rng{B}$ along
    $\brac*{\nullsp{A^*}\cap \rng{B}}+\nullsp{B^*}$. The last sum is an orthogonal decomposition, and $P$ is an orthogonal projection.
\end{lemma}

\begin{proof}
    \[P=A\pinv{A}\pinv{\brac*{B^*}}B^*=A\pinv{A}\brac*{B\pinv{B}}^*=A\pinv{A}B\pinv{B}.\]
    Therefore, $P$ is a projection and it is a product of two orthogonal projections. Moreover
    \[\rng{A\pinv{A}}=\rng{A},\quad \rng{B\pinv{B}}=\rng{B}.\]
    By Lemma~\ref{lem:prod_of_proj} $\rng{P}=\rng{A}\cap\rng{B}$.
    The preimage of $v\in\nullsp{A\pinv{A}}=\nullsp{A^*}$ under $B\pinv{B}$ is non--empty if and only if $v\in\rng{B\pinv{B}}=\rng{B}$. In such a case, the fiber is $v+\nullsp{B\pinv{B}}=v+\nullsp{B^*}$. Moreover, the range and the null space of $P$ are orthogonal.
\end{proof}

\begin{corollary}
    If the reverse order law $\pinv{(AB)}=\pinv{B}\pinv{A}$ holds, then $\dim\rng{B}\cap\rng{A^*}=\rank(AB)$.
\end{corollary}

The proof of the following Lemma is ommited.

\begin{lemma}\label{lem:prod_of_proj}
    Let $P,Q\in\bK^{n\times n}$ be matrices of two orthogonal projections, i.e., $P^2=P, Q^2=Q$ and $P^*=P,Q^*=Q$. Then
    \[PQv=v\Longleftrightarrow v\in\rng{P}\cap\rng{Q}.\]
    In particular, if $PQ$ is a projection, then $\rng{PQ}=\rng{P}\cap\rng{Q}$.
\end{lemma}

\begin{lemma}
    Let $A\in\bK^{m\times n},\ B\in\bK^{n\times k}$.
    Then \[\pinv{(AB)}=\pinv{B}\pinv{A}\Longleftrightarrow B\pinv{\brac{AB}}A\text{ is an orthogonal projection}\Longleftrightarrow \]
\[\Longleftrightarrow \rng{BB^*A^*}\oplus\nullsp{B^*A^*A}\text{ is an orthogonal decomp}\Longleftrightarrow \rng{BB^*A^*}=\rng{A^*AB}.\]

\end{lemma}

\begin{proof}
The last three equivalences are straightforward with the use of Lemma~\ref{lem:proj}. For the first equivalence $(\Longrightarrow)$ see Lemma~\ref{lem:rol_implies_proj} \\
    $(\Longleftarrow)$ matrix of a orthogonal projection is Hermitian, therefore
    \[B\pinv{\brac{AB}}A=A^*\pinv{\brac{B^*A^*}}B^*,\]
    pre-multiplying by $A$ and post-multiplying by $B$ gives
    \[AB=AA^*\pinv{\brac{B^*A^*}}B^*B,\]
    which implies the reverse order law by Tian~\cite[Theorem~11.1$\langle 3\rangle$]{Tian512}.
\end{proof}

%% file: generalized_wedd.tex
The following theorem generalizes classical Wedderburn rank reduction, cf.~Lemma~\ref{lem:classic_wedd}. The idea of proof is similar.

\begin{theorem}[generalized Wedderburn rank reduction formula]\label{thm:gen_wedd}
Let $A\in\bK^{m\times n},\ X\in\bK^{n\times p},\ Y\in\bK^{m\times q}$. Assume that $\rank Y^*AX=k$ (matrix $Y^*AX$ is possibly a non--square matrix). Then
\[\rank\brac*{A-(AX)\pinv{(Y^*AX)}(Y^*A)}=\rank A -k.\]
\end{theorem}

\begin{proof}
    Let $B=A-(AX)\pinv{(Y^*AX)}(Y^*A)$. Clearly, $\nullsp{A}\subset\nullsp{B}$.
    We need to find additional $k$ linearly independent vectors in $\nullsp{B}$. 
    Again, as $B(X\pinv{(Y^*AX)})=0$ 
    \[\rng{X\pinv{(Y^*AX)}}\subset \nullsp{B}.\]
    Note that 
    \[\rng{\pinv{(Y^*AX)}}=\rng{X^*A^*Y}\subset \rng{X^*},\]
    but $\rng{X^*}\cap \nullsp{X}=0$ therefore
    \[\rank X\pinv{(Y^*AX)}=k.\]
    We claim that 
    \[\rng{X\pinv{(Y^*AX)}}\cap \nullsp{A}=0.\]
    To show this assume that there exists $y\in\bK^q$ such that
    \[x=X\pinv{(Y^*AX)}y,\quad Ax=0,\]
    i.e. $x\in\bK^n$ lies in the intersection. Multiplying the first equation by $Y^*A$ on the left and applying the second one gives
    \[(Y^*AX)\pinv{(Y^*AX)}y=0.\]
    This means that $y$ lies in the kernel of the orthogonal projection onto $\rng{Y^*AX}$, i.e. 
    \[y\in\nullsp{(Y^*AX)^*}=\nullsp{\pinv{(Y^*AX)}}.\]
    Therefore $x=0$. To finish the proof it is enough to show that
    \[\nullsp{B}=\nullsp{A}\oplus\rng{X\pinv{(Y^*AX)}}.\]
    By the previous considerations, the right hand side is contained in the left one. Let $Bx=0$. That is
    \[A(x-X\pinv{(Y^*AX)}(Y^*A)x)=0,\]
    i.e. $x-X\pinv{(Y^*AX)}(Y^*A)x\in\nullsp{A}$ and the second term is in $\rng{X\pinv{(Y^*AX)}}$. This finishes the proof.
    \end{proof}
\begin{remark}
    The above theorem, with  matrices $X,Y$ suitably chosen, implies the generalized Nystr\"om decomposition, which implies the CUR decomposition.
\end{remark}
\begin{corollary}
    Let $X=A^*,Y=A$ then $\rank XAY^*=\rank A$ (images are orthogonal to kernels) and therefore
    \[A=AA^*\pinv{\brac*{A^*AA^*}}A^*A.\]
\end{corollary}
\begin{lemma}
    Let $B=A-(AX)\pinv{(Y^*AX)}(Y^*A)$. If $A$ is a square matrix and $A^2=A$ then $B^2=B$.
    That is reduction of a projection is a projection.
\end{lemma}
\begin{proof}
    \[B^2=A^2-2(AX)\pinv{(Y^*AX)}(Y^*A)+(AX)\pinv{(Y^*AX)}(Y^*A)(AX)\pinv{(Y^*AX)}(Y^*A)=B.\]
\end{proof}
\begin{corollary}\label{cor:red_of_As}
    If $B$ is a reduction of $A$ by $X,Y$ then $B^*$ is a reduction of $A^*$ by $Y,X$. Therefore, reduction of an orthogonal projection is an orthogonal projection.
\end{corollary}

\begin{lemma}
    Generalized Wedderburn rank reduction can be written in form
    \[B=A-PA=(I-P)A,\quad B=A-AQ=A(I-Q).\]
    where $P=(AX)\pinv{(Y^*AX)}Y^*$ is a rank $k$ projection onto $\rng{AXX^*A^*Y}$ along $\nullsp{X^*A^*YY^*}=\nullsp{\brac*{YY^*AX}^*}$
    and $Q=X\pinv{(Y^*AX)}Y^*A$ is a rank $k$ projection onto $\rng{XX^*A^*Y}$ along $\nullsp{X^*A^*YY^*A}$.
    
    Moreover
    \[\dim\rng{P}=\dim\rng{Q}=k,\quad \dim\nullsp{P}=\dim\nullsp{Q}=m-k.\]
    In particular 
    \[\rng{B}=\rng{A}\cap\nullsp{\pinv{\brac*{Y^*AX}}Y^*},\quad\nullsp{B}=\nullsp{A}\oplus
    \rng{X\pinv{(Y^*AX)}}.\]
    Note that 
    \[\rng{AXX^*A^*Y}=\rng{AX\pinv{(Y^*AX)}}.\]
\end{lemma}

\begin{proof}
    The subspace $\nullsp{B}$ was calculated in Theorem~\ref{thm:gen_wedd}. By Corollary~\ref{cor:red_of_As}, 
    \[\nullsp{B^*}=\nullsp{A^*}\oplus\rng{Y\pinv{(X^*A^*Y)}}.\]
    Therefore
    \[\rng{B}=\brac*{\nullsp{A^*}\oplus\rng{Y\pinv{(X^*A^*Y)}}}^\perp=\rng{A}\cap\nullsp{\pinv{(Y^*AX)}Y^*}.\]
\end{proof}

\begin{corollary}
     Matrix $A$ and its generalized Wedderburn reduction $B$ are equal (as linear transformations) when restricted to 
    \[\rng{I-Q}=\nullsp{Q}=\nullsp{X^*A^*YY^*A}.\]
\end{corollary}

\begin{lemma}[Ameli,Shadden~\cite{AmeliShadden}]
    Let $B=A-(AX)\pinv{(Y^*AX)}(Y^*A)$ be a rank reduction of matrix $A$. Then $B\pinv{A}B=\pinv{A}$, that is $B$ is a $\{2\}$-inverse of $\pinv{A}$. Moreover, $\pinv{A}B\pinv{A}$ is a reduction of $\pinv{A}$.
\end{lemma}
\begin{proof}
Multiplying $\pinv{A}$ on the left and on the right by $B$ gives 
\[B\pinv{A}B=A-(AX)\pinv{(Y^*AX)}(Y^*A)\pinv{A}A-A\pinv{A}(AX)\pinv{(Y^*AX)}(Y^*A)+\]
\[+(AX)\pinv{(Y^*AX)}(Y^*A)\pinv{A}(AX)\pinv{(Y^*AX)}(Y^*A)=B.\]
Moreover
    \[\pinv{A}B\pinv{A}=\pinv{A}-(\pinv{A}AX)\pinv{(Y^*A\pinv{A}AX)}(Y^*A\pinv{A})=\]
    \[=\pinv{A}-(\pinv{A}(AX))\pinv{((Y^*A)\pinv{A}(AX))}((Y^*A)\pinv{A}).\]
    Therefore $\pinv{A}B\pinv{A}$ is a reduction of $\pinv{A}$ by matrices $AX$ and $A^*Y$.
\end{proof}

\begin{lemma}
For  $A\in\bK^{m\times n}$
    let $X,X'\in\bK^{n\times p}$ and $Y,Y'\in\bK^{m\times q}$ be matrices of the same size. Then the generalized Wedderburn reduction of $A$
    with respect to $X,Y$ is equal to the generalized Wedderburn reduction of $A$
    with respect to $X',Y'$ if there exist $W_X,W_Y$ such that
    \[X-X'=\brac*{I-\pinv{A}A}W_X,\]
    \[Y-Y'=\brac*{I-A\pinv{A}}W_Y,\]
    that is the range of the difference of $X$ and $X'$ (resp. $Y$ and $Y'$) is contained in the nullspace of $A$ (resp. $A^*$).
\end{lemma}
\begin{proof}
Observe that $AX=AX'$ and $Y'^*A=Y^*A$ as
    \[AX-AX'=A\brac*{I-\pinv{A}A}W_X=0,\]
    \[Y^*A-Y'^*A=W_Y^*(I-A\pinv{A})A=0.\]
\end{proof}

%% file: wedd_decomp.tex
\begin{lemma}\label{lem:wedd_decomp}
Let $A\in\bK^{n\times k}$ and $X\in\bK^{n\times p},\ Y\in\bK^{m\times q}$. Assume that $\rank Y^*AX=\rank A.$ Then
\begin{equation}\label{eq:full_red}
A=(AX)\pinv{(Y^*AX)}(Y^*A).    
\end{equation}

Moreover
\[\pinv{A}=\pinv{(Y^*A)}(Y^* AX)\pinv{(AX)}.\]
\end{lemma}
\begin{proof}

The first claim follows directly from the generalized Wedderburn rank reduction, cf.~Theorem~\ref{thm:gen_wedd}, for $M=Y^*AX$ and
$k=\rank M=\rank A$. The second claim follows from two the reverse order law, by Corollary~\ref{cor:Greville_cond}, applied twice to the first claim, Eq.~(\ref{eq:full_red}). To see this note that
\[\rank (AX)\le \rank A,\quad \rank(Y^*A)\le \rank A,\quad \rank\pinv{\brac*{Y^*AX}}(Y^*A)\le \rank A\]
but by the assumption and the first claim~(\ref{eq:full_red}) neither of these inequalities may be strict, so
\[\rank(AX)=\rank \pinv{\brac*{Y^*AX}}=\rank\pinv{\brac*{Y^*AX}}(Y^*A)=\rank(Y^*A)=\rank A. \]

It suffices to show that

\[\nullsp{AX}=\nullsp{\brac*{\pinv{\brac*{Y^*AX}}(Y^*A)}^*},\quad\text{and}\quad
\nullsp{\pinv{\brac*{Y^*AX}}}=\nullsp{\brac*{Y^*A}^*}.\]
By several applications of Lemma~\ref{lem:null_equal_ranks} and the fact that $\nullsp{\pinv{M}}=\nullsp{M^*}$ for any matrix $M$
\[\nullsp{AX}=\nullsp{X},\]
\[\nullsp{\brac*{\pinv{\brac*{Y^*AX}}(Y^*A)}^*}=\nullsp{Y^*AX}=\nullsp{X},\]
and similarly
\[\nullsp{\pinv{\brac*{Y^*AX}}}=\nullsp{\brac*{Y^*A}^*}=\nullsp{Y}.\]
Therefore, by Corollary~\ref{cor:Greville_cond}
\[\pinv{A}=\pinv{\brac*{(AX)\pinv{(Y^*AX)}(Y^*A)}}=\pinv{(Y^*A)}\pinv{\brac*{\pinv{(Y^*AX)}(Y^*A)}}=\pinv{(Y^*A)}(Y^* AX)\pinv{(AX)}.\]
\end{proof}

Note the similarity to the~\cite[Theorem~5.5]{HammCU}.

\begin{corollary}
Let $P=(AX)\pinv{(Y^*AX)}Y^*$ and let $Q=X\pinv{(Y^*AX)}(Y^*A)$. Then $PAQ=PA=AQ=A$. Moreover $P,Q$ are oblique projections and
\[\rng{P}=\rng{A},\quad \nullsp{P}=\nullsp{Y^*}=\rng{Y},\]
\[\rng{Q}=\rng{X},\quad \nullsp{Q}=\nullsp{A}.\]
\end{corollary}

\begin{proof}
\[P^2=(AX)\pinv{(Y^*AX)}(Y^*AX)\pinv{(Y^*AX)}Y^*=AX\pinv{(Y^*AX)}Y^*=P,\]   
in a similar way $Q^2=Q$. The nullspaces can be calculated as in the proof of Lemma~\ref{lem:wedd_decomp}. For column spaces, inclusions follow from the definitions and from equal ranks.
\end{proof}

\begin{corollary}
    Any matrix can be presented as a product of itself and two projections, 
    \[A=PAQ,\]
    if $\rank (Y^*AX)=\rank A$ and 
    \[P=(AX)\pinv{(Y^*AX)}Y^*,\quad Q=X\pinv{(Y^*AX)}(Y^*A).\]
    
\end{corollary}
The Corollary is meta--factorization introduced by Karpowicz, see~\cite[Theorem~4]{karpowicz2022theory}. Moreover, any meta--factorization can be tautologically presented this way.

\begin{lemma}
    If $A=PAQ$ where $P\in\bK^{m\times m},Q\in\bK^{n\times n}$ are matrices of projections then there exist matrices $X,Y$ such that
     \[P=(AX)\pinv{(Y^*AX)}Y^*,\quad Q=X\pinv{(Y^*AX)}(Y^*A),\]
     and $\rank Y^*AX=\rank A$.
\end{lemma}

\begin{proof}
   If $A=PAQ$  then by left multiplication by $P$ and by right multiplication by $Q$ 
   \[AQ=PAQ=PA=A.\]
   In addition
   \[\rank P=\rank Q=\rank A,\]
   as $PA=A$ and $A^*=Q^*A^*$
   \[\rng{P}=\rng{A},\quad \rng{Q^*}=\rng{A^*}.\]
   Set $X=Q$ and $Y=P^*$. Then by Lemma~\ref{lem:proj} it follows that
   \[(AX)\pinv{\brac*{Y^*AX}}Y^*=(AQ)\pinv{\brac*{PAQ}}P=A\pinv{A}P=P,\]
   as $A\pinv{A}$ is a matrix of the orthogonal projection onto $\rng{A}=\rng{P}$. Similarly,
   \[X\pinv{\brac*{Y^*AX}}Y^*A=Q\pinv{\brac*{PAQ}}PA=Q\pinv{A}A=Q,\]
   as $\pinv{A}A$ is a matrix of the orthogonal projection onto $\rng{A^*}=\rng{Q^*}$ and
   the last equality is equivalent to $Q^*=\pinv{A}AQ^*$.
   
\end{proof}

%% file: best_rank.tex
The following lemma explains reduction of $A$ in terms of its SVD decomposition, for a particular choice of matrices $X$ and $Y$.
\begin{lemma}\label{lem:SVD}
    Let $A=U\Sigma V^*$ be a SVD decomposition of the matrix $A$. Let $r=\rank A$ and $I\subset\{1,\ldots,r\}$ be a non--empty subset.
    Let $X=V_I,Y=U_I$. Then
    \[AX\pinv{\brac*{Y^*AX}}Y^*A=\sum_{i\in I}\sigma_i u_i v_i^*.\]
    In particular, generalized Wedderburn reduction by such chosen matrices gives
    \[A-AX\pinv{\brac*{Y^*AX}}Y^*A=\sum_{i\in [r]\setminus I}\sigma_i u_i v_i^*.\]
\end{lemma}

\begin{proof}
Let $I=\{i_1,\ldots,i_k\}$. Then
\[Y^*AX=U_I^*AV_I=\brac*{\sum_{j\in [k]} e_{j}u_{i_j}^*}\brac*{\sum_{i\in [r]}\sigma_i u_i v_i^*}\brac*{\sum_{j\in [k]}v_{i_j} e_{j}^*}=\sum_{j\in [k]}\sigma_{i_j}e_j e_j^*.\]
In a similar fashion
\[\pinv{\brac*{Y^*AX}}=\sum_{i\in [k]}\sigma_{i_j}^{-1} e_j e_j^*,\]
\[AX\pinv{\brac*{Y^*AX}}Y^*X=
\brac*{\sum_{j\in [k]}\sigma_{i_j} u_{i_j} e_j^*}
\brac*{\sum_{i\in [k]}\sigma_{i_j}^{-1} e_j e_j^*} 
\brac*{\sum_{j\in [k]}\sigma_{i_j} e_j v_{i_j}^*}=\sum_{i\in [k]}\sigma_{i_j} u_{i_j} v_{i_j}^*=\sum_{i\in I}\sigma_i u_i v_i^*.\]
\end{proof}

\begin{corollary}\label{cor:red_by_SVD_part}
    Let $A=U\Sigma V^*$ be the SVD decomposition of matrix $A$. Let $r=\rank A$. Let $X\in\bK^{m\times p},\ Y\in\bK^{n\times q}$ be matrices such that $\rng{X}=\rng{V_I}$ and 
    $\rng{Y}=\rng{U_I}$. Then the generalized Wedderburn reduction of $A$ with respect to $X,Y$ is equal to
    \[A-AX\pinv{\brac*{Y^*AX}}Y^*A=\sum_{i\in [r]\setminus I}\sigma_i u_i v_i^*.\]
\end{corollary}
\begin{proof}
By the assumption, there exist full row rank matrices $M\in\bK^{\abs{I}\times p}$ and $N\in\bK^{\abs{I}\times q}$
such that
\[X=V_IM,\quad Y=U_IN.\]
Since $V_I^*$ is an orthogonal projection onto $\rng{V_I}$
\[U_I^*A=U_I^*AV_IV_I^*,\]
both sides are zero on $v_i$ when $i\notin I$ and identical otherwise.
Similarly
\[U_IU_I^*AV_I=AV_I.\]
Let $B=U_I^*AV_I=\sum_{j=1}^k \sigma_{i_j}e_j e_j^*$. Since $M$ is of full row rank and $N^*$ is of full column rank, $\rank (N^*BM)=\rank B$ and by Lemma~\ref{lem:wedd_decomp}
\[B=BM\pinv{\brac*{N^*BM}}N^*B.\]

Then
\[AX\pinv{\brac*{Y^*AX}}Y^*A=AV_IM\pinv{\brac*{N^*U_I^*AV_IM}}N^*U_I^*A=\]
\[=U_IU_I^*AV_IM\pinv{\brac*{N^*U_I^*AV_IM}}N^*U_I^*AV_IV_I^*=U_IBM\pinv{\brac*{N^*BM}}N^*BV_I^*=\]
\[=U_IBV_I^*=\sum_{i\in I}\sigma_i u_iv_i^*.\]
\end{proof}

\begin{remark}\label{rem:cancel2}
    In general, it is not true that the generalized Wedderburn reduction of fixed matrix $A$ by matrices $X$ and $Y$ depends only on $\rng{X}$ and $\rng{Y}$.
    On the other hand, when  $\rank AX=\rank A$ (for example if $X$ is of full row rank) and $Y^*$ is of full column rank 
    \[AX\pinv{(Y^*AX)}=A\pinv{(Y^*A)}.\]
    To see this, note 
    \[Y^*AX\pinv{(Y^*AX)}=Y^*A\pinv{(Y^*A)},\]
    as both sides are orthogonal projections on $\rng{Y^*A}=\rng{Y^*AX}$ and multiply the above equation on the left by $\pinv{\brac*{Y^*}}$. \\
    Similarly, if $\rank Y^*A=\rank A$ (for example $Y^*$ is of full column rank) and $X$ if of full row rank
    \[\pinv{(Y^*AX)}Y^*A=\pinv{(AX)}A.\]
    To see this, note 
    \[\pinv{(Y^*AX)}Y^*AX=\pinv{(AX)}AX,\]
    as both sides are orthogonal projections on $\rng{(AX)^*}=\rng{(Y^*AX)^*}$ and multiply the above equation on the left by $\pinv{\brac*{X}}$. \\
\end{remark}

\begin{example}
    The following example shows that the assumption that $Y$ is of full column rank is necessary.

\[A=\left[\begin{matrix}1 & 2 & 1\\2 & 3 & 2\\1 & 1 & 2\end{matrix}\right],\quad
Y=\left[\begin{array}{ccc} 1 & 4 & 1\\ 2 & 5 & 1\\ 3 & 6 & 1 \end{array}\right].
\]
Let $X=A$. Then $\rank AX=\rank A=3$, but $\rank Y=2$, so $Y$ is not of full rank and $AX\pinv{(Y^*AX)}Y^*A\neq A\pinv{(Y^*A)Y^*A}$. This also shows that the reduction with fixed $Y$ does not depend solely on the image of $\rng{X}$ (since $\rng{X}=\rng{I}$). By symmetry (by conjugation), the reduction with fixed $X$ does not depend solely on the image of $\rng{Y}$.
    
\end{example}
\begin{lstlisting}
A = [1 2 1; 2 3 2; 1 1 2]; 
Y = [1 2 3; 4 5 6; 1 1 1]'; % rank 2
X=A; rank(A*X)-rank(X)
A*X*pinv(Y'*A*X) - A*pinv(Y'*A) % not zero
\end{lstlisting}
\begin{example}
    The following example shows that the reduction does not depend on either $\rng{X}$ and $\rng{Y}$ or on only $\rng{X}$ and $\rng{Y^*}$.
\end{example}
\begin{lstlisting}
A = rand(3); X = rand(3); G=rand(3); % all invertible
Y = [1 2 3; 4 5 6; 1 1 1]; % rank 2 then the following are non-zero
norm(A*X*G*pinv(Y'*A*X*G)*Y'*A - A*X*pinv(Y'*A*X)*Y'*A) 
norm(A*G*X*pinv(Y'*A*G*X)*Y'*A - A*X*pinv(Y'*A*X)*Y'*A) 
norm(A*X*pinv(Y'*G*A*X)*Y'*G*A - A*X*pinv(Y'*A*X)*Y'*A) 

A = rand(3); Y = rand(3); G=rand(3); % all invertible
X = [1 2 3; 4 5 6; 1 1 1]; % rank 2
norm(A*X*pinv(G'*Y'*A*X)*G'*Y'*A - A*X*pinv(Y'*A*X)*Y'*A) % non-zero

\end{lstlisting}

%% file: misc.tex
In this section, (non)uniqueness of the presentation of a projection is discussed. Then, the properties shared by the matrix and its generalized Wedderburn rank reduction are listed (Remark~\ref{rem:inherited}). Finally, for two commuting projections, their meet and join are expressed in the form 

\begin{lemma}
Let $A\in\bK^{m\times p},B\in\bK^{m\times q}$ and $C\in\bK^{m\times s},D\in\bK^{m\times t}$. 
Let $S=AA^*BB^*,\ T=CC^*DD^*$.
Then
    \[A\pinv{(B^*A)}B^*=C\pinv{(D^*C)}D^*\]
    if and only if
    \[\rng{S}=\rng{T},\ \nullsp{S}=\nullsp{T},\quad\text{or equivalently}\quad
    \rng{S}=\rng{T},\quad \rng{S^*}=\rng{T^*},\]
    \[\text{or}\quad S\pinv{S}=T\pinv{T},\quad \pinv{S}S=\pinv{T}T.\]
\end{lemma}

\begin{proof}
    Two projections are equal if and only if they have the same range and null space, that is, by Lemma~\ref{lem:proj}  
    \[\rng{AA^*B}=\rng{CC^*D},\quad\nullsp{A^*BB^*}=\nullsp{C^*DD^*},\]
    or, since $\rng{B}=\rng{BB^*},\rng{D}=\rng{DD^*}$, by taking orthogonal complements
    \[\rng{AA^*BB^*}=\rng{CC^*DD^*},\quad\rng{BB^*AA^*}=\rng{DD^*CC^*},\]
    which finishes the proof.
\end{proof}

\begin{remark}
    Although any projection can be expressed in the form $A\pinv{(B^*A)}B^*$ (Corollary~\ref{cor:prescribed_proj}, this presentation is highly non-unique as when $C=AU$ and $D=BV$ for some orthogonal matrices $U,V$ then, by Lemma~\ref{lem:proj}, (the images and kernels are equal)
    \[A\pinv{(B^*A)}B^*=C\pinv{(D^*C)}D^*.\]
\end{remark}

\begin{remark}\label{rem:inherited}
    \begin{enumerate}[i)]
        \item generalized Wedderburn rank reduction of a Hermitian positive semidefinite matrix by $X=Y$ is
        Hermitian  positive semidefinite.
        Let $A^{\frac{1}{2}}$ be a Hermitian square root of $A$ and let $B=A^{\frac{1}{2}}X$. Then
        \[A-AX\pinv{(X^*AX)}(X^*A)=A^{\frac{1}{2}}\brac*{I-A^{\frac{1}{2}}X\pinv{(X^*A^{\frac{1}{2}}A^{\frac{1}{2}}X)}X^*A^{\frac{1}{2}}}A^{\frac{1}{2}}=\]
        \[=A^{\frac{1}{2}}\brac*{I-B\pinv{(B^* B)}B^*}A^{\frac{1}{2}}=
        A^{\frac{1}{2}}\brac*{I-B\pinv{B}}A^{\frac{1}{2}}.\]
        \item generalized Wedderburn rank reduction of a skew--Hermitian  matrix by $X=Y$ is
        a skew--Hermitian matrix.
        \[\brac*{A-AX\pinv{(X^*AX)}(X^*A)}^*=A^*-A^*X\pinv{(X^*A^*X)}(X^*A^*)=-\brac*{A-AX\pinv{(X^*AX)}(X^*A)}.\]
        
        \item generalized Wedderburn rank reduction of a normal matrix by $X=Y$ 
        in general does not need to be a normal matrix.
        Let
        \[A= \frac{1}{\sqrt{6}}\left[\begin{array}{rrr} \sqrt{3} & 0 & \sqrt{3}\\ -\sqrt{2} & \sqrt{2} & \sqrt{2}\\ 1 & 2 & -1 \end{array}\right],\quad X=\left[\begin{array}{c} 1\\ 0\\ 1 \end{array}\right].\]
        Then
        \[B=A-AX\pinv{(X^*AX)}(X^*A)=\frac{1}{\sqrt{6}}\left[\begin{array}{rrr} -1 & -2 & 1\\ -\sqrt{2} & \sqrt{2} & \sqrt{2}\\ 1 & 2 & -1 \end{array}\right],\quad\text{and}\quad [B,B^*]\neq 0.\]
        This example also shows that generalized Wedderburn rank reduction of a nondiagonalizable matrix by $X=Y$ generally does not need to be a non--diagonalizable matrix.
        \item Similarly, generalized Wedderburn rank reduction of a real diagonalizable over $\bR$ matrix  by $X=Y$ in general does not need to be a diagonalizable over $\bR$ matrix.
        \item the generalized Wedderburn rank reduction of a nilpotent matrix by $X=Y$ does not need to be nilpotent, take
        \[A=\left[\begin{array}{rrrr} 0 & 1 & 0 & 0\\ 0 & 0 & 1 & 0\\ 0 & 0 & 0 & 1\\ 0 & 0 & 0 & 0 \end{array}\right],\quad X=\left[\begin{array}{c} 1\\ 1\\ 1\\ 0 \end{array}\right].\]
        Then
        \[A-AX\pinv{(X^*AX)}(X^*A)=\frac{1}{2}\left[\begin{array}{rrrr} 0 & 1 & -1 & -1\\ 0 & -1 & 1 & -1\\ 0 & 0 & 0 & 2\\ 0 & 0 & 0 & 0 \end{array}\right],
\]
is not nilpotent.
    \item a version of Sherman--Morrison--Woodbury formula for pseudoinverse was proven by Deng in~\cite{Deng}, however, it seems unlikely that a concise formula for a reduction of a reduction exists.
    \end{enumerate}
\end{remark}

\begin{lemma}
    Let $P,Q$ be two commuting projections, that is, $PQ=QP$. Then $P\wedge Q:=PQ$ and $P\vee Q:=P+Q-PQ$ are projections such that
    \[\rng{P+Q-PQ}=\rng{P}+\rng{Q},\quad\nullsp{P+Q-PQ}=\nullsp{P}\cap \nullsp{Q},\]
    \[\rng{PQ}=\rng{P}\cap \rng{Q},\quad\nullsp{PQ}=\nullsp{P}+\nullsp{Q}.\]
\end{lemma}

\begin{proof}
    They are simultaneously diagonalisable, consider diagonal matrices. Then the claims are trivial.
\end{proof}

It is therefore reasonable to call $P\wedge Q$ the {\bf join} and $P\vee Q$ the {\bf meet} of $P$ and $Q$.

\begin{theorem}\label{thm:meet_and_join}
If projections $P,Q$ commute, then
\[P\vee Q=P+Q-PQ=S\pinv{\brac*{TS}}T,\]
where
\[S=\brac*{P+Q}\pinv{\brac*{P+Q}},\quad T=\pinv{\brac*{P+Q}}\brac*{P+Q}.\]
Moreovoer,
\[P\wedge Q=PQ=S\pinv{\brac*{TS}}T,\]
for
\[S=(PQ)\pinv{\brac*{PQ}},\quad T=\pinv{\brac*{PQ}}(PQ).\]
In particular, if $P,Q$ are orthogonal projections, then in both cases $S=T$ and
\[P\vee Q=\brac*{P+Q}\pinv{\brac*{P+Q}},\quad P\wedge Q=PQ=(PQ)\pinv{\brac*{PQ}}.\]
    
\end{theorem}
\begin{proof}
   For the first claim, since $\rng{P+Q}=\rng{P}+\rng{Q}$, as $P$ and $Q$ are postive semidefinite matrices,
   \[\rng{P\vee Q}=\rng{S},\quad \nullsp{P\vee Q}=\nullsp{P}\cap \nullsp{Q}=\brac*{\rng{P^*}+\rng{Q^*}}^\perp=\rng{T}^\perp.\]
   The rest follows from Lemma~\ref{lem:pinv_of_orth_projs} and Corollary~\ref{cor:prescribed_proj}. For the second, claim note that, by Lemma~\ref{lem:prod_of_proj}
   \[\rng{P\wedge Q}=\rng{S},\quad \nullsp{P\wedge Q}=\nullsp{P}+\nullsp{Q}=\brac*{\rng{P^*}\cap\rng{Q^*}}^\perp=\rng{T}^\perp.\]
   If $P,Q$ are orthogonal then $\rng{P}=\rng{P^*},\ \rng{Q}=\rng{Q^*}$ therefore $\rng{S}=\rng{T}$ and $S=T$.
\end{proof}

\begin{lemma}
    If
    \[\proj{A}{B}+\proj{C}{D}=I,\]
    then
    \[AA^*BB^*CC^*DD^*=CC^*DD^*AA^*BB^*=0.\]
\end{lemma}

\begin{proof}
    The image of one projection is equal to the kernel of the other. The claim follows from the 
    description of the image and the kernel of a projection in Lemma~\ref{lem:proj_dim}.
\end{proof}

The converse follows when the ranks of $B^*A$ and $C^*D$ are complementary.

%% file: matlab.tex
The following code illustrates the usage of the generalized Wedderburn rank reduction formula.

\begin{lstlisting}
m = 200; n = 30;
p = 10; q = 7;
r = 20; % rank of A

A = rand(m,r)*rand(r,n); % random m-by-n matrix of rank r
X = sprand(n,p,0.01); Y = sprand(m,q,0.01);

M = Y'*A*X; 
P = (A*X)*pinv(M)*Y'; Q = X*pinv(M)*(Y'*A); 
B = (A*X)*pinv(Y'*A*X)*(Y'*A);

rank(B) - (rank(A) - rank(M))

Y=Y*rand(q,2*q); % different Y with the same image
BB = (A*X)*pinv(Y'*A*X)*(Y'*A);
norm(B-BB) % different in general
\end{lstlisting}

The code below shows the correctness of Corollary~\ref{cor:red_by_SVD_part}.

\begin{lstlisting}
% see for yourself
m = 15; n = 13; r = 7;

A = rand(m,r) * rand(r,n); % random m-by-n matrix of rank r
[U, S, V] = svd(A);
k = randi([1 r-1],1,1); P = randperm(r); % the rank of A will be reduced by k
I = sort(P(1:k)); % choose k random columns
UI = U(:,I); VI = V(:,I);
p = randi([1 k],1,1); q = randi([1 k],1,1);
M = rand(k,k+p); N = rand(k,k+q);  % two full row rank matrices of k rows and random number of columns
X = VI * M;  Y = UI * N;  % matrices spanned by columns of V and U

k
svd((A*X)*pinv(Y'*A*X)*(Y'*A))
S_list = diag(S); % vector of singular values

% the generalized Wedderburn reduction of A by X and Y
% is the same as the k-rank reduction by SVD
norm((A*X)*pinv(Y'*A*X)*(Y'*A) - UI*S(I,I)*VI')

\end{lstlisting}